\documentclass[final,3p,12pt]{elsarticle}

\usepackage[english]{babel}   
\usepackage[T1]{fontenc}      
\usepackage[utf8]{inputenc}   
\usepackage{lmodern}          

\usepackage{lineno}
\usepackage{graphicx}
\usepackage{xcolor}
\usepackage{multirow}
\usepackage{xspace}
\usepackage{amsfonts}
\usepackage{caption}
\usepackage{amsmath}
\usepackage{amsthm}
\usepackage{amssymb}
\usepackage{comment}
\usepackage{enumitem}
\usepackage{mathrsfs}
\usepackage{tikz}
\usepackage{tikz-cd}

\newcommand{\R}{\mathbb{R}}
\newcommand{\C}{\mathbb{C}}
\newcommand{\N}{\mathbb{N}}

\newcommand{\Gr}{\textnormal{Gr}}
\renewcommand{\P}{\mathbb{P}}

\newcommand{\ovl}{\overline}

\renewcommand{\varnothing}{\emptyset}

\newtheorem{definition}{Definition}
\newtheorem{lemma}[definition]{Lemma}

\newtheorem{corollary}[definition]{Corollary}
\newtheorem{remark}[definition]{Remark}

\newtheorem{mylemmaA}{Lemma \textbf{A}.\!}

\newtheorem{mypropB}{Proposition \textbf{B}.\!}

\newtheorem{mycorB}{Corollary \textbf{B}.\!}

\newtheorem{mypropC}{Proposition \textbf{C}.\!}

\newtheorem{mytheoremD}{Theorem \textbf{D}.\!}

\theoremstyle{definition}
\newtheorem{example}[definition]{Example}

\begin{document}

\begin{frontmatter}

\title{Real Line Congruences of Trilinear Birational Maps}

\author[jkuag]{\footnote{The three authors contributed equally; the order of authors is alphabetical.}Bert J\"uttler\corref{cor}}
\ead{bert.juettler@jku.at}
\author[cunef]{Pablo Maz\'on}
\ead{pablo.mazon@cunef.edu}
\author[jkurisc]{Josef Schicho}
\ead{josef.schicho@jku.at}
\address[jkuag]{Institute of Applied Geometry, Johannes Kepler University,
Linz, Austria}
\address[cunef]{Department of Mathematics,
CUNEF Universidad,
Madrid, Spain}
\address[jkurisc]{Research Institute for Symbolic Computation, Johannes Kepler University,
Linz, Austria}
\cortext[cor]{Corresponding author}

\begin{abstract}
Trilinear mappings appear naturally when performing spatial
isogeometric discretizations of degree $p = 1$.  Among them,
birational mappings are characterized by the property that both the
mapping and the associated inverse mapping are rational and thus
easy to evaluate.  These mappings have recently been analyzed by
\cite{trilinear}.  Among other results, the authors provide a
classification of these mappings over the field of complex numbers.

The parameter lines of trilinear mappings form three two-parameter
systems of straight lines, and thus it is promising to analyze these
mappings with the tools provided by the field of line geometry, which
is a classical branch of higher geometry \citep{pottmann}.  Indeed, in
the birational case, the three systems of lines form space-filling
line congruences associated with rational mappings that can be used to
parameterize certain algebraic surfaces~\citep{usingLCs}.  Moreover,
the three systems are closely related, and based on these observations
we will present a geometric discussion of the results of
\cite{trilinear} together with a more detailed analysis of the
classification over the field of real numbers.
\end{abstract}


\end{frontmatter}

\section{Introduction}

Line geometry is a classical field that dates back to Julius Plücker’s
19th‑century work. Line geometry is based on the idea that the
fundamental entities that make up the space are the straight
lines. \citet{pottmann} provide a thorough and virtually complete
introduction to this fascinating branch of geometry.  This framework
is particularly well suited for geometric objects in which many lines
are present.

One-dimensional varieties of lines are known as ruled surfaces in
differential geometry; it is therefore natural to study such surfaces
using concepts from line geometry \citep{PEREZDIAZ2026578}.
Two-dimensional varieties of lines are called line congruences; they
arise naturally in contexts such as algebraic surface parameterization
\citep{jue:93b,usingLCs} and Hausdorff distance computation
\citep{Sohn2002236}.  Three-dimensional varieties of lines, which are
called line complexes, have also been studied.

Line geometry has numerous applications. These include the analysis
and the design of three-dimensional rigid body motions, since these
can be modeled as rotations (represented by dual unit quaternions) of
the dual unit sphere. Recent results include the analysis of rational
motions with low-degree trajectories \citep{MR4031400}, CNC motion
planning \cite{Calleja18}, and the analysis of line-symmetric motions
\citep{MR4932549}. Other applications include computer vision
\citep{Ronda2008193,MR4765605} and discrete differential and
combinatorial geometry \citep{MR3499552,MR2886550}.

Rational spline parameterizations--based on rational spline (i.e.,
NURBS) functions--are fundamental in Computer‑Aided Geometric Design
and underpin Isogeometric Analysis \citep[IGA,][]{Hughes20054135},
where the geometry map directly affects approximation spaces,
quadrature, conditioning, and solver performance.  Within IGA, but
also in other contexts, birational mappings are attractive because
they enable exact rational pull-back and push-forward of fields
defined by isogeometric discretizations. Prior work provides
constructive criteria and explicit inverses for planar and trivariate
cases, including bilinear quadrilateral maps \citep{Sederberg20151},
degree $(1, n)$ families \citep{Sederberg20161}, quadratic planar
birational maps \citep{Wang2021}, and a complete characterization of
trilinear birational mappings in 3D over the complex field
\citep{trilinear}.

Trilinear mappings appear naturally when performing spatial
isogeometric discretizations of degree $p=1$.  These parameterizations
are naturally connected to the classical and computational theory of
line congruences.  The parameter lines of trilinear mappings form
three two‑parameter varieties (i.e., congruences) of lines, making
line geometry a natural tool for analysis.  However, this connection
remains largely unexplored in the context of birational trilinear
maps.

Trilinear rational quaternionic parameterizations have been applied
by \cite{application_dupint} to construct ``Dupin cyclidic cubes''.
There, the parameter lines form three two-parameter families of straight
lines or circles. Moreover, circles from two different families
intersect only orthogonally (if they do intersect).

The present paper is devoted to the geometry of the line congruences
associated with all trilinear birational maps. More precisely, we
analyze these mappings through three parametric line congruences and
their interactions, with a focus on classification over the real
numbers, thereby extending the results of \cite{trilinear}.

The paper is organized as follows. 
In Section~\ref{sec: preliminaries} we recall the basic notions needed for the classification of line congruences. 
In particular, we briefly review the Plücker coordinates of lines in $\P_\C^3$ and the Klein quadric $Q\subset \P_\C^5$. 
We also introduce algebraic real line congruences, together with their focal varieties and the classes that will appear in the paper. 
Moreover, we present trilinear birational maps, whose parametric line congruences are the main object of study of this work.
In Section~\ref{sec: first results} we prove two preliminary results concerning the focal varieties of real line congruences and recall some facts about the geometry of complex lines in $\P_\C^3$. 
Finally, in Section~\ref{sec: classification} we develop the core of the paper. 
More precisely, we classify the line congruences arising from real trilinear birational maps, considering all possible types: $(1,1,1)$, $(1,1,2)$, $(1,2,2)$, and $(2,2,2)$ and their permutations. 
Each type is treated in a separate subsection, and all subsections follow the same structure, explained at the opening of the section.  

\section{Preliminaries}
\label{sec: preliminaries}

\subsection{Duality}
\label{subsec: duality}
Given a vector space $V$ of dimension $n$, we have the isomorphism 
\begin{equation}
\label{eq: duality}
\bigwedge^k V 
\simeq 
\bigwedge^{n-k} 
V^\vee 
\otimes 
\bigwedge^n V
\end{equation}
for each $1\leq k\leq n$, where $V^\vee$ is the dual vector space. 
Since $\dim \bigwedge^n V = 1$, the isomorphism above induces the identification $\P(\bigwedge^k V) \simeq \P(\bigwedge^{n-k} V^\vee)$. 

\subsection{Plücker coordinates of lines in 3D space}
\label{subsec: plucker coordinates}

Let $V = \C^4$. 
A line $\ell\subset \P_{\C}^3 = \P(\C^4)$ is defined by two points $a=(a_0:a_1:a_2:a_3)$ and $b=(b_0:b_1:b_2:b_3)$ in $\P_{\C}^3$. 
The Plücker coordinates $(c_{01}:c_{02}:c_{03}:c_{23}:c_{31}:
c_{12})\in\P_{\C}^5 = \P(\bigwedge^2\C^4)$ of $\ell$ are determined by the exterior product $a\wedge b$, namely 
\begin{gather*}
a\wedge b = (a_0 e_0 + a_1 e_1 + a_2 e_2 + a_3 e_3) 
\wedge 
(b_0 e_0 + b_1 e_1 + b_2 e_2 + b_3 e_3) 
= \\[2pt]
c_{01} (e_0 \wedge e_1) + 
c_{02} (e_0 \wedge e_2) + 
c_{03} (e_0 \wedge e_3) + 
c_{23} (e_2 \wedge e_3) + 
c_{31} (e_3 \wedge e_1) + 
c_{12} (e_1 \wedge e_2)
\end{gather*}
where $e_i$ is the $i$-th canonical vector in $\C^4$.  
In analogy, $\ell$ is determined by two planes $\alpha=(\alpha_0:\alpha_1:\alpha_2:\alpha_3)$ and $\beta=(\beta_0:\beta_1:\beta_2:\beta_3)$ in $(\P_{\C}^3)^\vee = \P((\C^4)^\vee)$. 
The Plücker coordinates $(\gamma_{01}:\gamma_{02}:\gamma_{03}:\gamma_{23}:\gamma_{31}:\gamma_{12})\in \P^5 = \P(\bigwedge^2(\C^4)^\vee)$ are determined by the exterior product $\alpha\wedge \beta$, namely 
\begin{gather*}
\alpha\wedge \beta = (\alpha_0 x_0 + \alpha_1 x_1 + \alpha_2 x_2 + \alpha_3 x_3) 
\wedge 
(\beta_0 x_0 + \beta_1 x_1 + \beta_2 x_2 + \beta_3 x_3) 
= \\[2pt]
\gamma_{01} (x_0 \wedge x_1) + 
\gamma_{02} (x_0 \wedge x_2) + 
\gamma_{03} (x_0 \wedge x_3) + 
\gamma_{23} (x_2 \wedge x_3) + 
\gamma_{31} (x_3 \wedge x_1) + 
\gamma_{12} (x_1 \wedge x_2)
\end{gather*}
where $x_i$ is the dual vector to $e_i$.  
In particular, the isomorphism $\P(\bigwedge^2 \C^4) \simeq \P(\bigwedge^2 (\C^4)^\vee)$ from \eqref{eq: duality} yields the identification between the Plücker coordinates, namely  
$$
(c_{01}:c_{02}:c_{03}:c_{23}:c_{31}:
c_{12})
=
(\gamma_{23}:\gamma_{31}:\gamma_{12}:\gamma_{01}:\gamma_{02}:\gamma_{03})
\ .
$$ 
The \textit{contraction} of $\bigwedge^2 \C^4$ by  $\gamma \in (\C^4)^\vee$ is the linear map $\iota_\gamma: \bigwedge^2 \C^4 \xrightarrow{} \C^4$ defined as 
$$
\iota_\gamma(a\wedge b) = \gamma(a) b - \gamma(b) a 
$$
on decomposable vectors and extended linearly.  
Geometrically, $\iota_\gamma(a\wedge b)$ determines the homogeneous coordinates of the intersection between the line $\ell\equiv a\wedge b$ and the plane $\gamma = 0$. 

\subsection{The Klein quadric}

The set of lines in $\P_{\C}^3$ is an algebraic variety $Q\subset \P_{\C}^5$ of dimension 4, known as the \textit{Klein quadric}. 
Namely, it is the image of the Plücker embedding 
\begin{eqnarray}
\label{eq: plucker embedding}
\iota : \Gr_2(\C^4) \simeq \Gr_2((\C^4)^\vee)
&\xrightarrow{} &
\P_{\C}^5 
\simeq 
\P(\C^4 \wedge \C^4)
\simeq 
\P((\C^4)^\vee \wedge (\C^4)^\vee)
\\[2pt] 
\nonumber
\ell 
&\mapsto 
& 
(c_{01}:c_{02}:c_{03}:c_{23}:c_{31}:
c_{12})
= 
(\gamma_{23}:\gamma_{31}:\gamma_{12}:\gamma_{01}:\gamma_{02}:\gamma_{01})
\end{eqnarray}
The defining equation of $Q$ is 
\begin{equation}
\ell \wedge \ell 
\equiv 
c_{01}c_{23}
+ 
c_{02}c_{31}
+
c_{03}c_{12}
= 
0
\ .
\end{equation}
Moreover, two lines $\ell,\ell'\subset \P_{\C}^3$ intersect if and only if 
\begin{equation}
\label{eq: intersecting lines}
\ell \wedge \ell'
\equiv 
c_{01}\gamma'_{23}
+ 
c_{02}\gamma'_{31}
+
c_{03}\gamma'_{12}
+
c_{23}\gamma'_{01}
+ 
c_{31}\gamma'_{02}
+
c_{12}\gamma'_{03}
= 
0
\ .
\end{equation}

\subsection{Real line congruences}
\label{subsection: LC}

An algebraic variety $X\subset \P_\C^n$ is \textit{real} if $X = \ovl{X}$, i$.$e$.$ it is invariant under complex conjugation. 
We denote by $X(\R)$ its real locus. 

\begin{definition}
\label{def: LC and real LC}
A(n algebraic) \textit{line congruence} is a(n algebraic) surface
$
X \subset Q \subset \P_{\C}^5
$. 
It is \textit{real} if $X$ is real and its real locus
$
X(\R)
$
has real dimension $2$.
\end{definition}

There are several ways to specify a line congruence $X\subset Q$. 
The two most straightforward approaches are (1) to provide the defining equations of $X$ in $\P_{\C}^5$, or (2) to provide a rational parametrization of $X$. 
An intermediate approach, which is often more geometrically intuitive, is (3) to provide a set of focal elements.  

\begin{definition}
\label{def: focal variety}
Let $X$ be a line congruence. 
A variety $V\subset \P_{\C}^3$ is \textit{focal} for $X$ if the following conditions are satisfied: 
\begin{enumerate}[itemsep=2pt]
\item[$a)$] $V\cap x \not=\varnothing$ for every $x \in X$.  
\item[$b)$] $V$ is minimal with respect to property $a)$, i$.$e$.$ if $W\subset V$ satisfies $a)$ then $W = V$.
\end{enumerate}  
\end{definition}

Most often, a line congruence is specified by focal curves. 
However, a single point can be a focal variety. 
In this paper, the following possibilities will appear. 

\begin{definition}
\label{def: types of LC}
Let $X$ be a real line congruence. 
We say that: 
\begin{enumerate}[itemsep=2pt]
\item[$A.$] $X$ is \textit{linear} if it is contained in a 3-plane in $\mathbb P^5_{\mathbb C}$. There are the following possibilities: 
\begin{itemize}
\item[$A.1.$] $X$ is \textit{hyperbolic linear} if it has two skew real focal lines.  
\item[$A.2.$] $X$ is \textit{elliptic linear} if it has two skew complex-conjugate focal lines.
\item[$A.3.$] $X$ is \textit{parabolic linear} if it has a single focal line (limit case of $A.1$ and $A.2$). 
\end{itemize}
\item[$B.$] $X$ is \textit{quadratic} if it has a focal plane conic and a focal line intersecting the conic.
\item[$C.$] $X$ is \textit{degenerate} if it has a single focal point\footnote{Degenerate congruences are also linear.}.
\end{enumerate}
\end{definition} 

It is helpful to view the focal elements of parabolic linear $(A.3)$ and degenerate $(C)$ line congruences as limiting configurations of two skew lines in $\P_\C^3$. 
From the perspective of algebraic geometry, such configurations are naturally related to flat limits, which are parameterized by the Hilbert scheme of two skew lines in $\P_\C^3$. 
We refer the reader to \cite[Section~2.1]{Hilbert_two_lines} and \cite[Example~6.3]{full_fiber} for further details. 
In particular, the following degenerations may occur:
\begin{enumerate}
\item Two skew lines $x$ and $y$ approach each other along a smooth quadric surface $Q\subset \P_\C^3$ until they become infinitely near (flat limit), yielding a double line contained in $Q$. 
The underlying scheme is supported on the line $x=y$, but carries additional infinitesimal structure. 
More precisely, to each point $P\in x$ there is attached a tangent direction contained in the tangent space $T_PQ$ of the quadric at $P$. 
Consequently, every line in the congruence passing through $P$ is tangent to $Q$. 
In this case we obtain a parabolic linear congruence $(A.3)$. 
We denote this limiting configuration by $y\leadsto x$, indicating that the line $y$ becomes infinitely near to $x$ through some smooth quadric $Q$.

\item Two skew lines $x$ and $y$ approach each other along a common normal direction until they intersect at a point. 
In this case the limiting scheme is the union of two coplanar lines $x\cup y$, with additional structure at the intersection point $P=x\cap y$. 
More precisely, $P$ becomes a spatial double point of $\P_\C^3$, meaning that every direction in space is attached to $P$. 
Accordingly, the limiting line congruence is degenerate: every line passes through $P$ and may take any direction in $\P_\C^3$.
\end{enumerate}

\subsection{Parametric line congruences of trilinear maps}
\label{subsec: tril bir}

In this paper, we are interested in the line congruences that arise as the parametric lines of trilinear birational maps. 
The coordinate ring of $(\P_\C^1)^3 = \P_\C^1 \times \P_\C^1 \times \P_\C^1$ is the tensor product $R = \C[s_0,s_1]\otimes \C[t_0,t_1]\otimes \C[u_0,u_1]$. 

\begin{definition}
\label{def: TBM}
A \textit{trilinear rational map} is a rational map 
\begin{eqnarray}
\label{eq: trilinear map}
\phi : (\P_\C^1)^3
&\dashrightarrow &\P_\C^3 
\\[2pt] 
\nonumber
(s_0:s_1)
\times 
(t_0:t_1)
\times 
(u_0:u_1)
&\mapsto 
& 
(f_0:f_1:f_2:f_3)
\end{eqnarray} 
for some $f_i = f_i(s_0,s_1,t_0,t_1,u_0,u_1)$ trilinear, without a common factor. 
It is \textit{real} if $f_0,f_1,f_2,f_3$ are all real polynomials. 
It is \textit{birational} if it admits an inverse rational map 
\begin{eqnarray}
\label{eq: inverse map}
\phi^{-1} : \P_\C^3 &\dashrightarrow & \P_\C^1 \times \P_\C^1 \times \P_\C^1
\\[2pt] 
\nonumber
(x_0:x_1:x_2:x_3)
&\mapsto 
&
(\sigma_0:
\sigma_1)
\times
(\tau_0:
\tau_1)
\times
(\upsilon_0:
\upsilon_1)
\end{eqnarray}
where $\sigma_i = \sigma_i(x_0,x_1,x_2,x_3)$ (resp$.$ for $\tau_j$ and $\upsilon_k$) are homogeneous of the same degree, without a common factor. 
\end{definition}

A trilinear rational map naturally yields three families of line congruences. 
Specifically, each of these line congruences consists of the parametric lines of the map. 

\begin{definition}
\label{def: parametric lines}
Let $\phi$ be a trilinear rational map. 
An \textit{$s$-line} is the image of the restriction to $(y_0:y_1)\times (z_0:z_1) = (\beta_0:\beta_1)\times (\gamma_0:\gamma_1)$ for some point in $\P_\C^1\times \P_\C^1$. 
The $t$- and \textit{$u$-lines} are defined analogously. 
\end{definition}

\begin{definition}
\label{def: parametric LCs}
Let $\phi$ be a trilinear rational map. 
The \textit{parametric line congruences} of $\phi$ are the line congruences $S,T,U$ respectively defined by the $s$-, $t$- and $u$-lines of the map. 
\end{definition}

\begin{remark}
\label{remark: no planar LCs}
\textit{Planar line congruences}, that is, congruences whose lines are all contained in a single plane, cannot arise as the parametric line congruences of trilinear birational maps. 
Indeed, for such a map to be dominant, the lines of the congruence must sweep out a Zariski open subset of $\P_\C^3$. 
\end{remark}

In general, it is always possible to parameterize a parametric line congruence of $\phi$ with biquadratic polynomials. 
More explicitly, the image of the map $\omega:\P_\C^1\times\P_\C^1\dashrightarrow Q$ given by 
\begin{equation}
\label{eq: Plucker coords of s-line}
(t_0:t_1)\times (u_0:u_1)
\mapsto
\Phi_S 
= 
\phi
(
(1:0)
\times 
(t_0:t_1)
\times 
(u_0:u_1)
)
\wedge 
\phi
(
(0:1)
\times 
(t_0:t_1)
\times 
(u_0:u_1)
)
\end{equation}
is $S$. 

The idea underlying \eqref{eq: Plucker coords of s-line} is to take the exterior product of two points on each parametric line. 
Since the original parametrization has trilinear entries in the variables $s_0,s_1,t_0,t_1,u_0,u_1$, the induced parametrization of the corresponding lines in Plücker coordinates is given by biquadratic polynomials in $t_0,t_1,u_0,u_1$.
However, when $\phi$ is birational, one can obtain parameterizations of lower degree by exploiting the syzygies, or moving planes, of $\phi$. 
More precisely, for each parametric line the syzygies yield two planes containing that line. 
Taking the exterior product of these two planes recovers the same Plücker coordinates as in \eqref{eq: Plucker coords of s-line}, as described in Section \ref{subsec: plucker coordinates}, but often with a reduced degree and more geometric content. 
This observation constitutes the key idea underlying our approach to the classification. 
Lastly, there exist several classes of trilinear birational maps.  
\begin{definition}
The \textit{type} of a trilinear
 birational map $\phi$ is $(\deg \sigma_i,\deg \tau_j, \deg \upsilon_k)\in \N^3$. 
\end{definition}

In \cite{trilinear}, the authors prove that the only possible types of a trilinear birational map are (up to permutation of the indices): $(1,1,1)$, $(1,1,2)$, $(1,2,2)$, and $(2,2,2)$.  

\section{Preliminary Results}
\label{sec: first results}

\subsection{Real Focal Varieties}

We begin by proving several results about the focal varieties of real line congruences that will be used later.

\begin{lemma}
\label{lemma: minimality of focal}
Let $X$ be a line congruence. 
Moreover, let $C\subset\P_\C^3$ be a curve such that $C\cap x\not= \varnothing$ for every $x\in X$. 
Then, one of the following holds: 
\begin{enumerate}
\item $V$ is a focal variety of $X$. 
\item $X$ is degenerate. 
\end{enumerate}
If additionally $X$ is real, then $\ovl{V}$ is also a focal variety of $X$. 
\end{lemma}

\begin{proof} 
If $V$ is not a focal variety of $X$, then condition $b)$ from Definition \ref{def: focal variety} fails. 
In particular, there is a variety $W \subsetneq V$ satisfying that $W\cap x\not=\varnothing$ for every $x\in X$. 
Since the only subvarieties (irreducible) of a curve are its points, $W$ must be a point in $V$. 
As any additional incidence condition does not yield a line congruence, the statement follows. 

Regarding the last statement, if $X$ is real then $X = \ovl{X}$. 
Hence, every line $x \in X$ is $x = \ovl{y}$ for some $y\in X$. 
In particular, every $x\in X$ intersects $\overline{V}$ since by definition $V\cap y \not= \varnothing$ for every $y \in X$. 
On the other hand, the minimality of $V$ implies the same property for $\overline{V}$. 
Therefore, $\ovl{V}$ is also a focal variety. 
\end{proof}

The next result shows that, among the possibilities listed in Definition \ref{def: types of LC}, nonreal focal varieties may arise only in the elliptic linear $(A.2)$ case. 

\begin{corollary}
\label{corollary: focal real}
The focal varieties of a real line congruence $X$ of class either $A.3$, $B$, or $C$ are real. 
\end{corollary}

\begin{proof}
Real line congruences of classes $A.3$ and $C$ have only one focal variety $V$, so Lemma \ref{lemma: minimality of focal} readily implies $\ovl{V} = V$. 
For class $B$, $V$ is either a plane conic or a line. 
Since $\ovl{V}$ is the same curve as $V$ and there is only one of each, we conclude $\ovl{V} = V$. 
\end{proof}

\subsection{Real Linear Congruences}

Next, we recall two classical results on the geometry of real lines in complex 3-space. 
Since these results are well known, we state them without proof.

\begin{lemma}
\label{lemma: types of complex conjugate lines}
Let $\ell,\ell'\subset \P_\C^3$ be two complex-conjugate lines. 
Then, one of the following holds:  
\begin{enumerate}
\item The lines $\ell,\ell'$ are equal, hence both real. 
\item The lines $\ell,\ell'$ are skew and have no real points. 
\item The lines $\ell,\ell'$ intersect at their unique real point. 
\end{enumerate}
\end{lemma}

\begin{lemma}
\label{lemma: skew lines}
Let $\ell,\ell'\subset \P_\C^3$ be skew lines. 
Then, the line congruence 
$$
X 
\equiv 
\{ x \subset \P_\C^3 : x \simeq \P_\C^1 ,\,  \ell\cap x \not= \varnothing, \ell'\cap x\not= \varnothing\}
$$
is real if and only if one of the following holds: 
\begin{enumerate}
\item $\ell,\ell'$ are both real (hyperbolic linear $(A.1)$). 
\item $\ell,\ell'$ are complex-conjugate (elliptic linear $(A.2)$). 
\end{enumerate}
\end{lemma}

\section{Classification of Real Line Congruences}\label{sec: classification}

In this section we present a complete classification of the real parametric line congruences arising from a real trilinear birational map. 
We proceed systematically through all possible types, namely $(1,1,1)$, $(1,1,2)$, $(1,2,2)$, and $(2,2,2)$. 
For each type, the analysis follows the same scheme. 
First, we give explicit formulas for the syzygies of the defining polynomials of general trilinear birational maps and use them to obtain explicit parameterizations of the parametric line congruences. 
Next, we use these parameterizations to describe the focal varieties of the general classes. 
Finally, we consider all possible degenerations in order to complete the classification.

\subsection{Real line congruences of trilinear birational maps of type $(1,1,1)$} 

In this section, we focus on real trilinear birational maps of type $(1,1,1)$, that is, maps whose inverse is defined by polynomials that are linear in each factor of $\mathbb{P}^1_{\mathbb{C}} \times \mathbb{P}^1_{\mathbb{C}} \times \mathbb{P}^1_{\mathbb{C}}$. 
Only for this type, we do not need to assume that $\phi$ is general. 
The reason is that we can provide explicit parameterizations of the syzygies in both the general and degenerate cases. 

We begin by providing an explicit parametrization of the associated parametric line congruences which arise from the syzygies of the map. 

\begin{mypropB} 
\label{propB: (1,1,1)}
Let $\phi$ be a trilinear birational map of type $(1,1,1)$. 
Then, the parametric line congruences admit the parameterizations $\P_\C^1\times\P_\C^1 \dashrightarrow Q$ given by  
\begin{eqnarray}
\label{eq: lc 1 (1,1,1)}
S \equiv &  
(t_0 \, B_0 + t_1 \, B_1)
\wedge 
(u_0 \, C_0 + u_1 \, C_1)
\ ,
\\
\label{eq: lc 2 (1,1,1)}
T \equiv &  
(s_0 \, A_0 + s_1 \, A_1)
\wedge 
(u_0 \, C_0 + u_1 \, C_1)
\ ,
\\
\label{eq: lc 3 (1,1,1)}
U \equiv  &  
(s_0 \, A_0 + s_1 \, A_1)
\wedge 
(t_0 \, B_0 + t_1 \, B_1)
\ \phantom{,}
\end{eqnarray}
for some $A_i, B_j , C_k \in (\C^4)^\vee$ such that any four are linearly independent\footnote{Here, $E = \{A_0,A_1,B_0,B_1,C_0,C_1\}$ is isomorphic to the uniform matroid $U_{4,6}$.}. 
\end{mypropB}

\begin{proof} 
By \cite[Proposition $6.2$]{trilinear}, the defining polynomials $\Phi = (f_0,f_1,f_2,f_3)$ of a trilinear birational map of type $(1,1,1)$ admit linear syzygies 
\begin{equation}
\label{eq: syzygies (1,1,1)}
\sigma =  s_0 \, A_0 + s_1 \, A_1 \ \ ,\ \ 
\tau =  t_0 \, B_0 + t_1 \, B_1 \ \ ,\ \
\upsilon = u_0 \, C_0 + u_1 \, C_1
\end{equation}
for some $A_i, B_j, C_k\in (\C^4)^\vee$ with any four linearly independent.  
Since the syzygies $\tau$ and $\upsilon$ do not depend on the variables $s_0,s_1$, it follows that the $s$-line $\ell_{p}$ associated to $p = (\beta_0:\beta_1)\times (\gamma_0:\gamma_1) \in (\P_\C^1)^2$ lies in the two planes $\tau(\beta_0,\beta_1) = 0$ and $\upsilon(\gamma_0,\gamma_1) = 0$. 
Therefore, its Plücker coordinates are $\ell_p\equiv \tau\wedge\upsilon$, and \eqref{eq: lc 1 (1,1,1)} follows. 
With the obvious modifications, the parameterizations of $T$ and $U$ follow as well.
\end{proof}

\begin{mypropC} 
\label{propC: (1,1,1)}
Let $\phi$ be a general trilinear birational map of type $(1,1,1)$. 
Then, there are unique pairwise skew lines $a,b,c\subset\P_\C^3$ such that: 
\begin{enumerate}[itemsep=2pt]
\item The focal lines of $S$ are $b$ and $c$. 
\item The focal lines of $T$ are $a$ and $c$. 
\item The focal lines of $U$ are $a$ and $b$. 
\end{enumerate}
\end{mypropC}

\noindent See also Fig.~\ref{fig:716}.

\begin{figure}\centering\setlength{\unitlength}{1mm}\footnotesize
  \begin{picture}(0,0)
    \put(12,3){$a$}\put(5,13){$b$}\put(10,45){$c$}
    \put(8,23){$S$}\put(80,22){$T$}\put(56,20){$U$}
  \end{picture}
  \includegraphics[width=.5\textwidth]{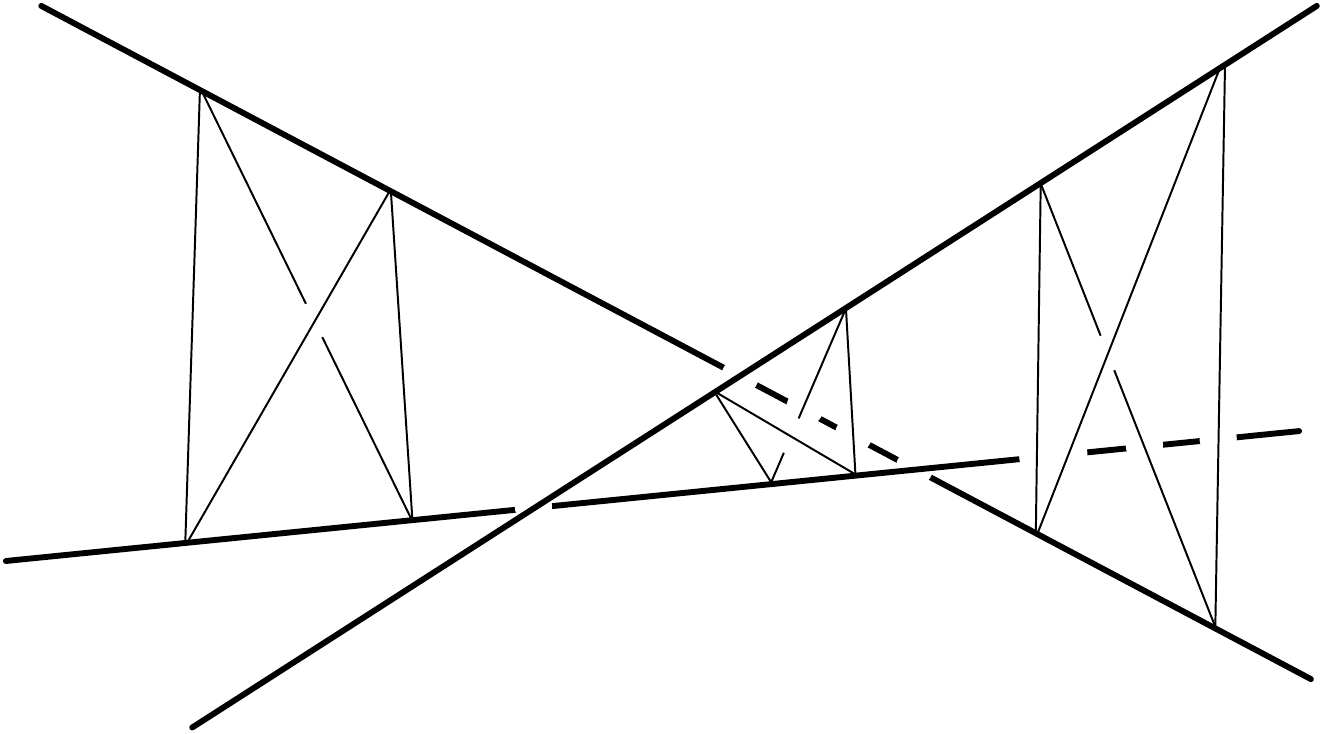}
  \caption{Focal lines (thick lines) and line congruences (thin lines) for type $(1,1,1)$, class 1.}
  \label{fig:716}
\end{figure}
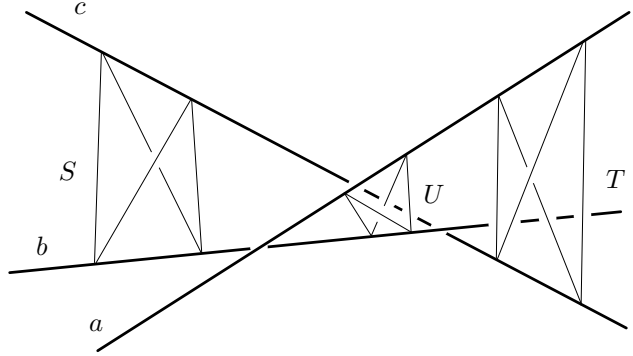

\begin{proof}
We start from the parameterizations \eqref{eq: lc 1 (1,1,1)}, \eqref{eq: lc 2 (1,1,1)} and \eqref{eq: lc 3 (1,1,1)} of Proposition B$.$\ref{propB: (1,1,1)}. 
Consider the lines $a,b,c\subset \P_\C^3$ with Plücker coordinates
\begin{equation}
\label{eq: focal lines (1,1,1)}
a \equiv A_0\wedge A_1 
\ , \ 
b \equiv B_0\wedge B_1
\ , \ 
c \equiv C_0\wedge C_1
\ . 
\end{equation}
The three exterior products 
$$
a\wedge b
\ \ ,\ \ 
a\wedge c
\ \ ,\ \ 
b\wedge c
$$
are nonzero, since any four vectors among $A_i,B_j,C_k$ are linearly independent. 
Therefore, the lines $a,b,c$ are pairwise skew. 

We now prove that every $\ell\in S$ meets both $b$ and $c$. 
In particular, $S$ is linear with focal lines $b$ and $c$. 
Recall from \eqref{eq: intersecting lines} that $\ell$ intersects $b$ if and only if $b \wedge \ell = 0$.  
In particular, the parametrization \eqref{eq: lc 1 (1,1,1)} yields 
\begin{equation*}
b \wedge S
= 
(B_0 \wedge B_1) 
\wedge 
(t_0 \, B_0 + t_1 \, B_1)
\wedge 
(u_0 \, C_0 + u_1 \, C_1)
= 0
\end{equation*}
as each term in the expansion repeats either $B_0$ or $B_1$.  
Similarly, it follows that $c\wedge S = 0$. 
Therefore, $b$ and $c$ intersect every line in $S$. 
The remaining two statements for $T$ and $U$ are derived from the same arguments.  
\end{proof}

The following is the classification of the parametric line congruences of real trilinear birational maps of type $(1,1,1)$.

\begin{mytheoremD}  
\label{theoD: (1,1,1)}
Let $\phi$ be a general real trilinear birational map of type $(1,1,1)$, and let $a,b,c\subset \P_\C^3$ be the lines in Proposition C$.$\ref{propC: (1,1,1)}. 
Then, the three lines are real. 

Moreover, every (possibly degenerate) real parametric line congruence falls (up to permutation of the factors of $\P_\C^1\times\P_\C^1\times\P_\C^1$) into one of the following classes:
\begin{enumerate}[itemsep=2pt]
\item (General) The lines $a,b,c$ are pairwise skew:  
\begin{enumerate}[itemsep=2pt]
\item $S$ is hyperbolic linear $(A.1)$ with focal lines $b$ and $c$. 
\item $T$ is hyperbolic linear $(A.1)$ with focal lines $a$ and $c$. 
\item $U$ is hyperbolic linear $(A.1)$ with focal lines $a$ and $b$. 
\end{enumerate} 
\item The lines $a,b$ intersect at a point, and $c$ is skew to both:   
\begin{enumerate}[itemsep=2pt]
\item $S$ is hyperbolic linear $(A.1)$ with focal lines $b$ and $c$. 
\item $T$ is hyperbolic linear $(A.1)$ with focal lines $a$ and $c$. 
\item $U$ is degenerate $(C)$ with a single focal point  $a\cap b$. 
\end{enumerate} 
\item The lines $a,b$ are skew, and $c$ intersects both at a point: 
\begin{enumerate}[itemsep=2pt]
\item $S$ is degenerate $(C)$ with a single focal point $b\cap c$. 
\item $T$ is degenerate $(C)$ with a single focal point $a\cap c$. 
\item $U$ is hyperbolic linear $(A.1)$ with focal lines $a$ and $b$. 
\end{enumerate} 
\item The three lines $a,b,c$ are distinct and coplanar: 
\begin{enumerate}[itemsep=2pt]
\item $S$ is degenerate $(C)$ with a single focal point $b\cap c$. 
\item $T$ is degenerate $(C)$ with a single focal point $a\cap c$. 
\item $U$ is degenerate $(C)$ with a single focal point $a\cap b$. 
\end{enumerate} 
\end{enumerate} 
\end{mytheoremD} 

\begin{proof}
If $\phi$ is general, Proposition~C.\ref{propC: (1,1,1)} readily implies that the three parametric line congruences are linear $(A)$ with pairwise skew focal lines $a$, $b$ and $c$. 
We first prove that $a,b,c$ are real. 
Suppose that $a$ is non-real. 
Since $a$ is a focal line of both $T$ and $U$, Corollary~\ref{corollary: focal real} implies that $\overline{a}$ is also a focal line of the two line congruences. 
As the line $c$ (resp.~$b$) is also a focal line of $T$ (resp.~$U$), we obtain $c=\overline{a}$ (resp.~$b=\overline{a}$). 
This yields a contradiction with the skewness of $b$ and $c$. 
Therefore, the three lines $a,b,c$ are real. 

Next, we proceed with the classification. 
Proposition~C.\ref{propC: (1,1,1)} already describes the general class. 
Since all parametric line congruences in the general class have only focal lines, the following degenerations may occur: 
\begin{enumerate}
\item Two focal lines approach each other along a common normal direction, until they intersect at a point. 
In this case, we obtain a degenerate line congruence $(C)$.
\item Two focal lines approach each other along a smooth quadric until they become infinitely near (flat limit), yielding a double line contained in a smooth quadric. 
In this case, we obtain a parabolic linear congruence $(A.3)$.
\end{enumerate} 
However, the degeneration described in $2$ cannot occur. 
Indeed, if the focal lines $a$ and $b$ degenerate to a double line on a smooth quadric, then the line congruences $S$ and $T$ would have the same focal lines, and hence would coincide.
The remaining degenerations are precisely those listed in the statement. 
The fact that all of them occur as parametric line congruences of a trilinear birational map of type $(1,1,1)$ follows from \cite[Theorem~5.1]{trilinear}.
\end{proof}

\subsection{Real line congruences of trilinear birational maps of type $(1,1,2)$}

Next, we focus on real trilinear birational maps of type $(1,1,2)$. 
As before, aim to give an explicit parametrization of the parametric line congruences, which again relate to the syzygies of the birational map.
To this goal, we first describe the syzygies of a general map of this type. 
Unlike Proposition B$.$\ref{propB: (1,1,1)}, we now require the assumption that the map is general. 

\begin{mylemmaA}
\label{lemmaA: (1,1,2)}
Let $\phi$ be a general trilinear birational map of type $(1,1,2)$. 
The defining polynomials $\Phi = (f_0,f_1,f_2,f_3)$ admit syzygies of the form 
\begin{gather}
\label{eq: syzygies (1,1,2)}
\sigma =  a_0 \, A - a_1 \, C_0 \ \ ,\ \ 
\tau =  b_0 \, B - b_1 \, C_0 \ \ \\[4pt]
\gamma = a_1 c_0 \, B - a_0 c_0 \, C_0 + a_0 c_1 \, C_1
\ \ ,\ \ 
\delta = b_1 c_0 \, A - b_0 c_0 \, C_0 - b_0 c_1 \, C_1 \ \ .
\end{gather} 
for some linear $a_i = a_i(s_0,s_1)$, $b_j = b_j(t_0,t_1)$, $c_k = c_k(s_0,s_1) \in R$ 
and linearly independent $A,B,C_0,C_1\in (\C^4)^\vee$.  
\end{mylemmaA}

\begin{proof}
It follows from \cite[Corollary $4.9$]{trilinear_siaga} that we can find linearly independent $A,B,C_0,C_1\in (\C^4)^\vee$ satisfying 
\begin{equation}
\langle 
A
, 
\Phi
\rangle
= 
a_1 b_0 c_1
\ \ ,\ \  
\langle 
B
, 
\Phi
\rangle
= 
a_0 b_1 c_1
\ \ ,\ \  
\langle 
C_0
, 
\Phi
\rangle
= 
a_0 b_0 c_1
\ \ ,\ \ 
C_1
= 
(a_0 b_0 - a_1 b_1) c_0
\end{equation}
for some linear $a_i=a_i(s_0,s_1)$, 
$b_j = b_j(t_0,t_1)$, 
$c_k = c_k(u_0,u_1)$. 
Hence, the syzygies follow. 
\end{proof}

\begin{mycorB} 
\label{corB: (1,1,2)}
With the notation of Lemma A$.$\ref{lemmaA: (1,1,2)}, the parametric line congruences admit the rational parameterizations $\P_\C^1\times\P_\C^1 \dashrightarrow Q$ given by  
\begin{eqnarray}
\label{eq: lc 1 (1,1,2)}
S \equiv &  
\tau \wedge \delta & = 
(b_0 \, B - b_1 \, C_0) 
\wedge 
(- b_0 c_0 \, A + b_1 c_0 \, C_0 + b_0c_1 \, C_1)
\ ,
\\[4pt]
\label{eq: lc 2 (1,1,2)}
T \equiv &  
\sigma \wedge \gamma & = 
(a_0 \, A - a_1 \, C_0) 
\wedge 
(a_0 c_0 \, B - a_1 c_0 \, C_0 + a_0 c_1 \, C_1) 
\ ,
\\[4pt]
\label{eq: lc 3 (1,1,2)}
U \equiv &  
\sigma \wedge \tau & = 
(a_0 \, A - a_1 \, C_0) 
\wedge 
(b_0 \, B - b_1 \, C_0) 
\ .
\end{eqnarray}
\end{mycorB}

\begin{proof}
The argument is the same as in the proof of Proposition~B.\ref{propB: (1,1,1)}. 
Namely, since the syzygies $\tau$ and $\delta$ do not depend on $s_0,s_1$, the $s$-line associated to $(t_0:t_1)\times (u_0:u_1) \in (\P_\C^1)^2$ lies in the two planes $\tau=0$ and $\delta=0$. 
Hence, its Plücker coordinates are given by $\tau\wedge\delta$. 
The remaining parameterizations follow by obvious modifications.
\end{proof}

\begin{mypropC} 
\label{propC: (1,1,2)}
Let $\phi$ be a general trilinear birational map of type $(1,1,2)$. 
Then, there are distinct lines $a,b\subset\P_\C^3$ and a plane conic $c\subset\P_\C^3$ satisfying:
\begin{enumerate}[itemsep=2pt]
\item The lines $a,b$ are coplanar. 
\item The intersections $P = a\cap c$ and $Q = b \cap c$ are points.
\end{enumerate}
such that: 
\begin{enumerate}[itemsep=2pt]
\item The focal curves of $S$ are $a$ and $c$.
\item The focal curves of $T$ are $b$ and $c$. 
\item $U$ has a single focal point $O = a\cap b$. 
\end{enumerate} 
\end{mypropC}

\noindent See also Fig.~\ref{fig:949}.

\begin{figure}\centering\setlength{\unitlength}{1mm}\footnotesize
  \begin{picture}(0,0)
    \put(20,52){$a$}\put(34,52){$b$}\put(0,5){$c$}
    \put(12,2){$Q$}\put(47,2){$P$}\put(32,46){$O$}
    \put(34,20){$S$}\put(21,10){$T$}\put(27,26){$U$}
  \end{picture}
  \includegraphics[width=.4\textwidth]{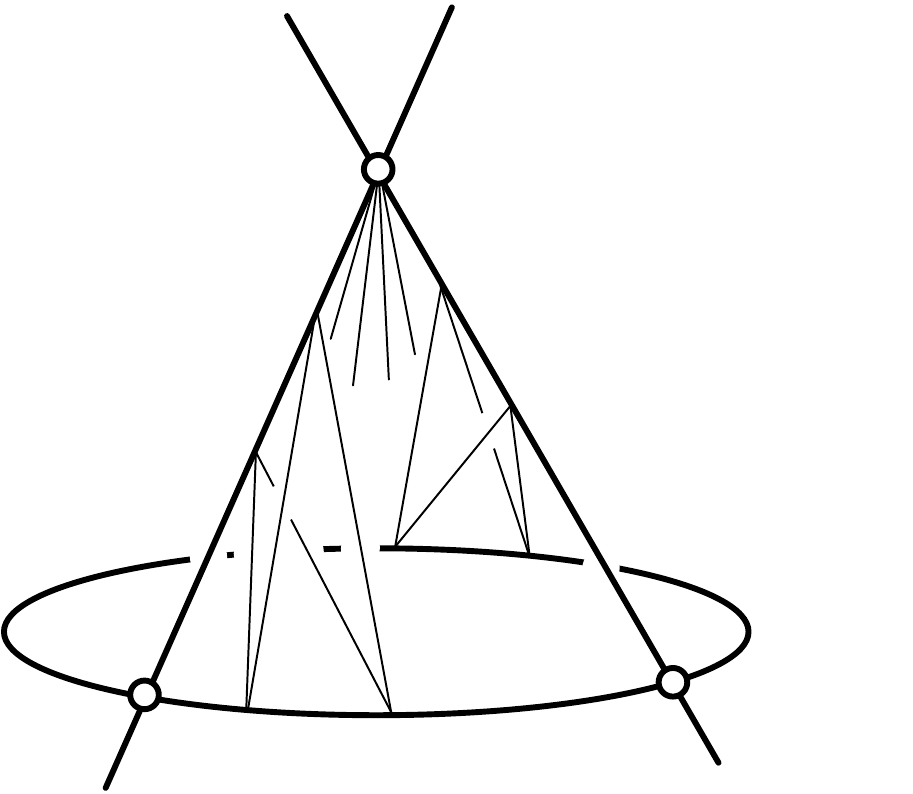}
  \caption{Focal curves (thick lines) and line congruences (thin lines) for type $(1,1,2)$, class 1.}
  \label{fig:949}
\end{figure}
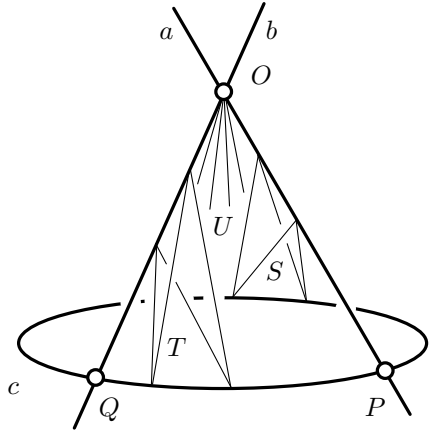

\begin{proof}
We start from the parameterizations \eqref{eq: lc 1 (1,1,2)}, \eqref{eq: lc 2 (1,1,2)} and \eqref{eq: lc 3 (1,1,2)} of Proposition C$.$\ref{propC: (1,1,2)}. 
First, consider the lines $a,b\subset \P_\C^3$ with Plücker coordinates
\begin{equation}
\label{eq: focal curves (1,1,2)}
a \equiv B \wedge C_0 
\ , \ 
b \equiv A \wedge C_0
\ ,
\end{equation} 
which are nonzero since $A,B$ and $C_0$ are linearly independent. 
It is straightforward that $a,b$ lie in the plane $C_0 = 0$. 
Moreover, $a$ (resp$.$ $b$) intersects every line $\ell\in S$ (resp$.$ $\ell\in T$), since $a\wedge S = 0$ (resp$.$  $b\wedge T = 0$) because each term in the expansion repeats either $B$ (resp$.$ $A$) or $C_0$.
Similarly, both $a$ and $b$ intersect every $\ell\in U$ since $a\wedge \ell = b \wedge \ell = 0$. 
Therefore, the point $O = a\cap b$ lies in every line $\ell\in U$, and the line congruence $U$ is degenerate with focal point $O$. 

Next, we prove that every $\ell\in S$ intersects a plane conic $c\subset\P_\C^3$. 
More specifically, $c$ is the conic defined by the plane $C_1 = 0$ and the quadric $AB-C_0^2 = 0$. 
The fact that every $\ell\in T$ intersects the same conic $c$ will follow from the same argument. 
Moreover, since $C_0,C_1$ are linearly independent the curves $a,b,c$ will not be  coplanar. 
First, since incidence is preserved by automorphisms of $\P_\C^3$, for simplicity we apply the linear change of variables 
$$
A \mapsto x_0 
\ ,\ 
B \mapsto x_1 
\ ,\ 
C_0 \mapsto x_2 
\ ,\ 
C_1 \mapsto x_3 
\ ,
$$
which sends the parametrization \eqref{eq: lc 1 (1,1,2)} of $S$ to 
\begin{equation*} 
S
\equiv 
b_0b_1
(- c_0 \, (x_0\wedge x_1) 
+ c_1 \, (x_2 \wedge x_3)) 
-
b_0^2
(c_0 \, 
(x_1 \wedge x_2)
+ 
c_1 \, 
(x_1 \wedge x_3)
)
+ 
b_1^2 
u_0 
\, 
(x_0\wedge x_2)
\in 
\bigwedge^2 (\C^4)^\vee
\ .
\end{equation*}
In these coordinates, $c$ is defined by $x_3 = x_0x_1 - x_2^2 = 0$. 
On the other hand, the isomorphism $\P(\bigwedge^2 (\C^4)^\vee) \simeq \P(\bigwedge^2 \C^4)$ (see Section \ref{subsec: plucker coordinates}) identifies 
these Plücker coordinates with 
\begin{equation} 
\label{eq: primal coordinates (1,1,2)}
S
\equiv 
b_0b_1
(- c_0 \, (e_2\wedge e_3) 
+ c_1 \, (e_0 \wedge e_1)) 
-
b_0^2
(c_0 \, 
(e_0 \wedge e_3)
- 
c_1 \, 
(e_0 \wedge e_2)
)
- 
b_1^2 
u_0 
\, 
(e_1\wedge e_3)
\in 
\bigwedge^2 
\C^4 
\end{equation} 
where $e_i$ is the dual vector to $x_i$. 
Now, let $\ell\in S$ be the generic $s$-line. 
We can compute the intersection $\ell \cap \{x_3 = 0\}$ as the contraction of $S$ in \eqref{eq: primal coordinates (1,1,2)} by $x_3$, obtaining 
\begin{equation}
\iota_{x_3}(S) 
= 
-
c_0
(
b_0^2 \, e_0 + b_1^2 \, e_1 + b_0b_1 \, e_2
) 
\ .
\end{equation}
In particular, $\iota_{x_3}(S)$ satisfies  $x_0x_1 - x_2^2 = 0$. 
Therefore, $\ell\cap\{x_3=0\}\in c$ and the plane conic $c$ is a focal curve of $c$.  
Additionally, since $A,B,C_0,C_1$ are linearly independent, $P = a\cap b \equiv A\wedge B \wedge C_0$ does not belong to the plane $C_1 = 0$, hence neither to $c$. 
\end{proof}

The following is the classification of the parametric line congruences of real trilinear birational maps of type $(1,1,2)$. 

\begin{mytheoremD}  
\label{theoD: (1,1,2}
Let $\phi$ be a general real trilinear birational map of type $(1,1,2)$, and let $a,b,c\subset\P_\C^3$ be  the lines and plane conic in \textup{Proposition C$.$\ref{propC: (1,1,2)}}. 
Then, the three curves are real. 

Recall that $O = a\cap b$. 
Then, every (possibly degenerate) real parametric line congruence falls (up to permutation of the factors of $\P_\C^1\times\P_\C^1\times\P_\C^1$) into one of the following classes:
\begin{enumerate}
\item (General) The conic $c$ is smooth and $O\not\in C$:  
\begin{enumerate}[itemsep=2pt]
\item $S$ is quadratic $(B)$ with focal curves $b$ and $c$. 
\item $T$ is quadratic $(B)$ with focal curves $a$ and $c$.  
\end{enumerate} 
\item The conic $c$ is smooth and $O\in C$:  
\begin{enumerate}[itemsep=2pt]
\item $S$ is quadratic $(B)$ with focal curves $b$ and $c$. 
\item $T$ is quadratic $(B)$ with focal curves $a$ and $c$.
\end{enumerate}
\item The conic $c=x\cup y$ is reducible and the four lines are distinct:  
\begin{enumerate}[itemsep=2pt]
\item $S$ is hyperbolic linear $(A.1)$ with focal lines $b$ and $x$. 
\item $T$ is hyperbolic linear $(A.1)$ with focal lines $a$ and $y$. 
\end{enumerate}
\item The conic $c=x\cup y$ is reducible and  $y\leadsto a$:  
\begin{enumerate}[itemsep=2pt]
\item $S$ is hyperbolic linear $(A.1)$ with focal lines $b$ and $x$. 
\item $T$ is parabolic linear $(A.3)$ with focal line $a$. 
\end{enumerate}
\item The conic $c=x\cup y$ is reducible, $y\leadsto a$ and $x\leadsto b$: 
\begin{enumerate}[itemsep=2pt]
\item $S$ is parabolic linear $(A.3)$ with focal line $2b$.
\item $T$ is parabolic linear $(A.3)$ with focal line $2a$.
\end{enumerate}
\end{enumerate}
\end{mytheoremD} 

\begin{proof}
If $\phi$ is general, Proposition C$.$\ref{propC: (1,1,2)} implies that $S$ and $T$ are both quadratic $(B)$ with focal curves $a,b,c$. 
Hence, the three curves must be real by Corollary \ref{corollary: focal real}. 
  
Next, we address the classification. 
Proposition C$.$\ref{propC: (1,1,2)} already provides the general class.  
On the other hand, since the focal curves are one conic and two coplanar lines, the following degenerations may occur:
\begin{enumerate}
\item The plane supporting $a\cup b$ becomes tangent to $c$. 
In this case, $S$ and $T$ remain quadratic line congruences $(B)$. 
\item The conic $c = x\cup y$ splits as the union of two lines $x,y$. 
In this case, $S$ and $T$ no longer have a smooth plane conic as a focal variety. 
Hence, they degenerate to linear congruences $(A)$. 
\item Two skew focal lines approach each other along a smooth quadric until they become infinitely near (flat limit), yielding a double line contained in a smooth quadric. 
In this case, we obtain a parabolic linear congruence $(A.3)$.
\end{enumerate} 
The possible degenerations are those listed in the statement. 
The fact that all these possibilities occur as the parametric line congruences of a trilinear birational map of type $(1,1,2)$ is ensured by \cite[Theorem $5.4$]{trilinear}.
\end{proof}

\subsection{Real Line Congruences of Trilinear Birational Maps of Type $(1,2,2)$} 

In this section, we focus on real trilinear birational maps of type $(1,2,2)$. 
Now, the inverse is quadratic in the last two factors of $\mathbb{P}^1_{\mathbb{C}} \times \mathbb{P}^1_{\mathbb{C}} \times \mathbb{P}^1_{\mathbb{C}}$, and linear in the first one. 

As usual, we begin providing explicit formulas for the syzygies of a general birational map. 

\begin{mylemmaA}
\label{lemmaA: (1,2,2)}
Let $\phi$ be a general trilinear birational map of type $(1,2,2)$. 
Then, the defining polynomials $\Phi = (f_0,f_1,f_2,f_3)$ admit syzygies of the form 
\begin{gather}
\alpha = 
a_0 \, A_1 - a_1 \, A_0
\ \ ,\ \ 
\beta = a_0 c_0 \, C_1 - a_1 c_1 \, C_0 
\ \ ,\ \ 
\gamma = a_0 b_0 \, B_1 - a_1 b_1 \, B_0 \ \ ,\\[4pt]
\delta_1 = b_2 c_0 \, B_0 - b_0 c_2 \, C_0 
\ \ ,\ \ 
\delta_2 = b_2 c_1 \, B_1 - b_1 c_2 \, C_1 
\end{gather} 
for some linear $a_i = a_i(s_0,s_1)$, $b_j = b_j(t_0,t_1)$, $c_k = c_k(s_0,s_1) \in R$ 
and $A_i,B_j,C_k\in \C^4$ where 
$$
E_0 \sqcup E_1
= 
\{
A_0
, 
B_0 
, 
C_0
\}
\sqcup 
\{
A_1
, 
B_1 
, 
C_1
\}
$$
verifies\footnote{These three conditions mean that the set $E_0 \sqcup E_1$ with linearly independent sets of vectors is the direct sum $U_{2,3} \oplus U_{2,3}$, where $U_{2,3}$ is the uniform matroid with $3$ elements and independent sets of size at most $2$.}: 
\begin{enumerate}
\item the set $E_i$ is linearly dependent;
\item any two vectors in $E_i$ are linearly independent;
\item two distinct $u_0,v_0\in E_0$ and two distinct $u_1,v_1\in E_1$ form a basis of $\C^4$.
\end{enumerate}
\end{mylemmaA}

\begin{proof}
It follows from \cite[Corollary $5.9$]{trilinear_siaga} that we can find $A_i,B_j,C_k\in (\C^4)^\vee$ satisfying the linear independence conditions in the statement together with 
\begin{equation}
\langle 
A_i 
, 
\Phi
\rangle
= 
a_i h 
\ \ ,\ \  
\langle 
B_i 
, 
\Phi
\rangle
= 
a_i b_i c_2 
\ \ ,\ \  
\langle 
C_i 
, 
\Phi
\rangle
= 
a_i b_2 c_i
\end{equation}
for some linear $a_i=a_i(s_0,s_1)$, 
$b_j = b_j(t_0,t_1)$, 
$c_k = c_k(u_0,u_1)$ and bilinear $h = h(t_0,t_1;u_0,u_1)$. 
Hence, the syzygies in the statement follow. 
\end{proof}

\begin{mycorB}
\label{corB: (1,2,2)}
With the notation of Lemma A$.$\ref{lemmaA: (1,2,2)},  the parametric line congruences admit the rational parameterizations $\P_\C^1\times\P_\C^1 \dashrightarrow Q$ given by  
\begin{eqnarray}
\label{eq: lc 1 (1,2,2)}
S \equiv &  
\delta_1 \wedge \delta_2 & = 
(b_2 c_0 \, B_0 - b_0 c_2 \, C_0 )
\wedge 
(b_2 c_1 \, B_1 - b_1 c_2 \, C_1 )
\ ,
\\[4pt]
\label{eq: lc 2 (1,2,2)}
T \equiv &  
\alpha \wedge \beta & = 
(a_0 \, A_1 - a_1 \, A_0) 
\wedge 
(a_0 c_0 \, C_1 - a_1 c_1 \, C_0 ) 
\ ,
\\[4pt]
\label{eq: lc 3 (1,2,2)}
U \equiv &  
\alpha \wedge \gamma & = 
(a_0 \, A_1 - a_1 \, A_0) 
\wedge 
(a_0 b_0 \, B_1 - a_1 b_1 \, B_0)
\ .
\end{eqnarray}
\end{mycorB}

\begin{mypropC} 
\label{propC: (1,2,2)}
Let $\phi$ be a general trilinear birational map of type $(1,2,2)$. 
Then, we can find unique lines $a,b,c,x,y\subset\P_\C^3$ satisfying:
\begin{enumerate}[itemsep=2pt]
\item[1.] $x,y$ are skew. 
\item[2.] $a,b,c$ are pairwise skew. 
\item[3.] All $a,b,c$ intersect $x$ (resp$.$ $y$).  
\end{enumerate}
such that:
\begin{enumerate}[itemsep = 2pt]
\item[a)] The focal lines of $S$ are $x$ and $y$. 
\item[b)] The focal lines of $T$ are $a$ and $c$. 
\item[c)] The focal lines of $U$ are $a$ and $b$. 
\end{enumerate}
\end{mypropC}

\noindent See also Fig.~\ref{fig:1241}.

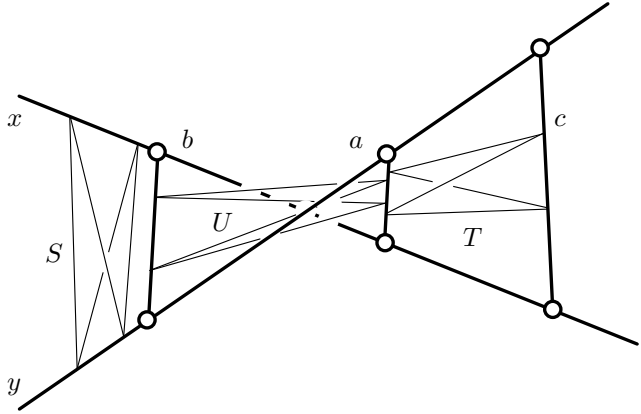
\begin{figure}\centering\setlength{\unitlength}{1mm}\footnotesize
  \begin{picture}(0,0)
    \put(45,35){$a$}\put(23,35){$b$}\put(72,38){$c$}
    \put(0,38){$x$}\put(0,3){$y$}
    \put(5,20){$S$}\put(60,22){$T$}\put(27,24){$U$}
  \end{picture}
  \includegraphics[width=.5\textwidth]{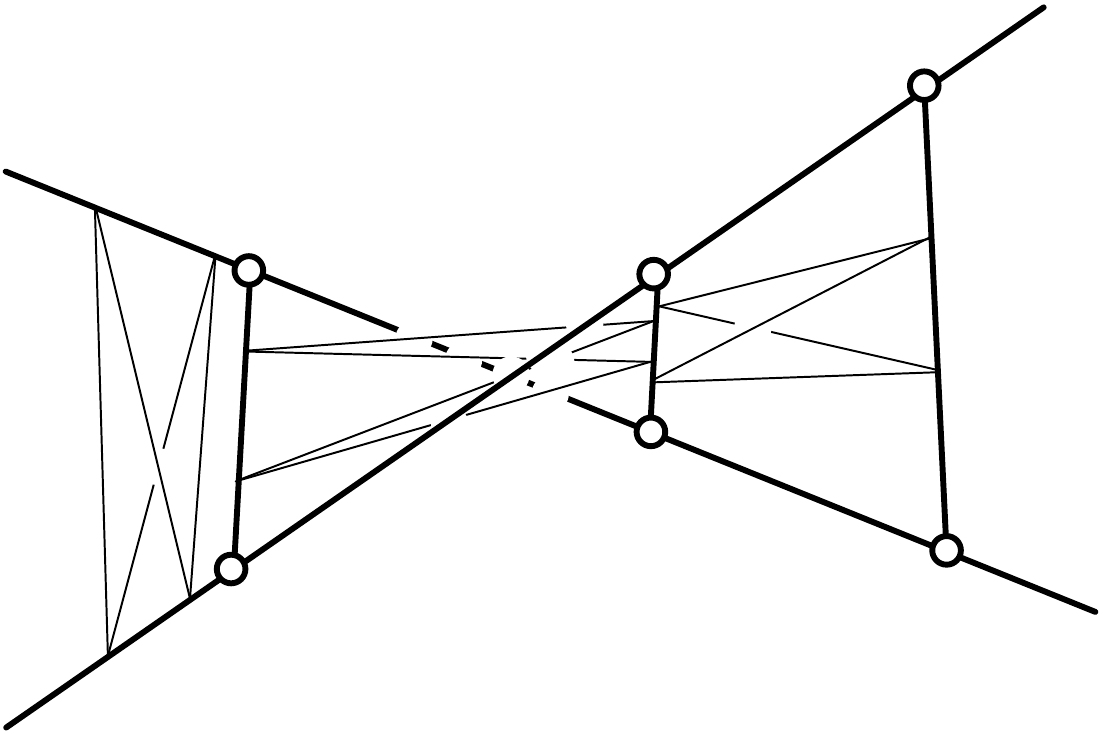}
  \caption{Focal lines (thick lines) and line congruences (thin lines) for  type $(1,2,2)$, class a1.}
  \label{fig:1241}
\end{figure}

\begin{proof}
We start from the parameterizations of Corollary B$.$\ref{corB: (1,2,2)}. 
First, define the lines $a,b,c,x,y\subset \P_\C^3$ with Plücker coordinates
\begin{equation}
\label{eq: focal lines (1,2,2)}
a \equiv A_0 \wedge A_1 
\ , \ 
b \equiv B_0 \wedge B_1
\ , \ 
c \equiv C_0 \wedge C_1
\ ,\ 
x \equiv B_0 \wedge C_0 
\ ,\ 
y \equiv B_1 \wedge C_1 
\ . 
\end{equation} 
The linear independence conditions in the statement of Lemma A$.$\ref{lemmaA: (1,2,2)} ensure the following:
\begin{enumerate}
\item Since $B_0,C_0\in E_0$ and $B_1,C_1\in E_1$ form a basis of $\C^4$, $x\wedge y = (B_0\wedge C_0) \wedge (B_1\wedge C_1) \not= 0$. 
Hence, $x$ and $y$ are skew. 
\item Since $A_0,B_0\in E_0$ and $A_1,B_1\in E_1$ form a basis of $\C^4$, $a\wedge b = (A_0\wedge A_1) \wedge (B_0\wedge B_1) \not= 0$. 
Hence, $a$ and $b$ are skew. 
The analogous arguments prove that $a,b,c$ are pairwise skew.
\item Since $E_0$ (resp$.$ $E_1$) is linearly dependent $(A_0\wedge A_1) \wedge (B_0\wedge C_0) = 0$ (resp$.$ $(A_0\wedge A_1) \wedge (B_1\wedge C_1) = 0$). 
Hence, $a\wedge x = a\wedge y = 0$ and $a$ intersects both $x$ and $y$. The analogous arguments prove that $b$ and $c$ also intersect both $x$ and $y$. 
\end{enumerate}

Next, it is straightforward that $a$ intersects every line $\ell\in T$ (resp$.$ $\ell\in U$) since $a\wedge T = 0$ (resp$.$ $a\wedge U = 0$) because $a \wedge \alpha = a \wedge (a_0 \, A_1 - a_1 \, A_0) = 0$. 
Similarly, $c$ (resp$.$ $b$) intersects every line $\ell\in T$ (resp$.$ $\ell\in U$) since $c\wedge T = 0$ (resp$.$ $b\wedge U = 0$). 
Therefore, $T$ (resp$.$ $U$) must be a linear congruence with focal lines $a$ and $c$ (resp$.$ $a$ and $b$). 
Finally, the same arguments yield every $\ell\in S$ intersects both $x$ and $y$. 
Hence, $S$ is also a linear congruence with focal lines $x$ and $y$. 
\end{proof}

\begin{mytheoremD} 
\label{theoD: (1,2,2)}
Let $\phi$ be a general real trilinear birational map of type $(1,2,2)$, and let $a,b,c,x,y\subset \P_\C^3$ be the lines in Proposition C$.$\ref{propC: (1,2,2)}. 
Then, one of the following possibilities holds:
\begin{itemize}
\item[a)] The lines $a,b,c,x,y$ are all real. 
\item[b)] The lines $a,b,c$ are real, and the lines $x,y$ are skew complex-conjugate. 
\end{itemize}
Moreover, every (possibly degenerate) real parametric line congruence in case a) falls (up to permutation of the factors of $\P_\C^1\times\P_\C^1\times\P_\C^1$) into one of the following classes: 
\begin{enumerate}
\item[a.1)] (General) The lines $a,b,c$ are pairwise skew and the lines $x,y$ are skew: 
\begin{itemize}
\item $S$ is hyperbolic linear $(A.1)$ with focal lines $x$ and $y$. 
\item $T$ is hyperbolic linear $(A.1)$ with focal lines $a$ and $c$. 
\item $U$ is hyperbolic linear $(A.1)$ with focal lines $a$ and $b$. 
\end{itemize}
\item[a.2)] The lines $a,c$ are skew, $b\leadsto a$, and the lines $x,y$ are skew: 
\begin{itemize}
\item $S$ is hyperbolic linear $(A.1)$ with focal lines $x$ and $y$. 
\item $T$ is hyperbolic linear $(A.1)$ with focal lines $a$ and $c$. 
\item $U$ is parabolic linear $(A.3)$ with focal line $2a$.  
\end{itemize}
\item[a.3)] The lines $a,b,c$ are pairwise skew and $y\leadsto x$:  
\begin{itemize}
\item $S$ is parabolic linear $(A.3)$ with focal line $2x$.  
\item $T$ is hyperbolic linear $(A.1)$ with focal lines $a$ and $c$. 
\item $U$ is hyperbolic linear $(A.1)$ with focal lines $a$ and $c$.  
\end{itemize}
\item[a.4)] The lines $a,c$ are skew, $b\leadsto a$, and $y\leadsto x$: 
\begin{itemize}
\item $S$ is parabolic linear $(A.3)$ with focal line $2x$.  
\item $T$ is hyperbolic linear $(A.1)$ with focal lines $a$ and $c$. 
\item $U$ is parabolic linear $(A.3)$ with focal line $2a$.   
\end{itemize} 
\item[a.5)] The line $a$ is skew with both $b,c$, the intersection $P = b\cap c \cap x$ is a point, and the lines $x,y$ are skew:
\begin{itemize}
\item $S$ is hyperbolic linear $(A.1)$ with focal lines $x$ and $y$.  
\item $T$ is hyperbolic linear $(A.1)$ with focal lines $a$ and $c$. 
\item $U$ is hyperbolic linear $(A.1)$ with focal lines $a$ and $b$. 
\end{itemize}
\item[a.6)] $b\leadsto a$ and $c\leadsto b$  approaching through different smooth quadrics, and $x,y$ are skew:  
\begin{itemize}
\item $S$ is hyperbolic linear $(A.1)$ with focal lines $x$ and $y$.  
\item $T$ is parabolic linear $(A.3)$ with focal line $2a$.  
\item $U$ is parabolic linear $(A.3)$ with focal lines $2a$.  
\end{itemize}
\item[a.7)] $b\leadsto a$ and $c\leadsto b$  approaching through different smooth quadrics, and $y\leadsto x$:  
\begin{itemize}
\item $S$ is parabolic linear $(A.3)$ with focal lines $2x$.  
\item $T$ is parabolic linear $(A.3)$ with focal line $2a$.  
\item $U$ is parabolic linear $(A.3)$ with focal lines $2a$.  
\end{itemize}
\item[a.8)] The line $a$ is skew with both $b,c$, the intersection $P = b\cap c\cap x$ is a point, and $y\leadsto x$:
\begin{itemize}
\item $S$ is parabolic linear $(A.3)$ with focal line $2x$.  
\item $T$ is hyperbolic linear $(A.1)$ with focal lines $a$ and $c$.   
\item $U$ is hyperbolic linear $(A.1)$ with focal lines $a$ and $b$. 
\end{itemize} 
\end{enumerate}

Furthermore, every (possibly degenerate) real parametric line congruence in case b) falls (up to permutation of the factors of $\P_\C^1\times\P_\C^1\times\P_\C^1$) into one of the following classes: 
\begin{enumerate}
\item[b.1)] The lines $a,b,c$ are pairwise skew: 
\begin{itemize}
\item $S$ is elliptic linear $(A.2)$ with focal lines $x$ and $y$.  
\item $T$ is hyperbolic linear $(A.1)$ with focal lines $a$ and $c$.   
\item $U$ is hyperbolic linear $(A.1)$ with focal lines $a$ and $b$. 
\end{itemize}
\item[b.2)] The lines $a,c$ are skew, and $b\leadsto a$:  
\begin{itemize}
\item $S$ is elliptic linear $(A.2)$ with focal lines $x$ and $y$.  
\item $T$ is parabolic linear $(A.3)$ with focal line $2a$.    
\item $U$ is hyperbolic linear $(A.1)$ with focal lines $a$ and $b$. 
\end{itemize}
\item[b.3)] $b\leadsto a$ and $c\leadsto a$ approaching through different smooth quadrics: 
\begin{itemize}
\item $S$ is elliptic linear $(A.2)$ with focal lines $x$ and $y$.  
\item $T$ is parabolic linear $(A.3)$ with focal line $2a$.    
\item $U$ is parabolic linear $(A.3)$ with focal line $2a$. 
\end{itemize}
\end{enumerate}
\end{mytheoremD} 

\begin{proof}
If $\phi$ is general, Proposition~C.\ref{propC: (1,2,2)} readily implies that the three parametric line congruences are linear $(A)$. 
We first prove that the lines $a,b,c$ are real. 
Suppose that $a$ is non-real. 
Then Corollary~\ref{corollary: focal real} implies that $\overline{a}$ is also a focal line of both $T$ and $U$. 
Since the line $c$ (resp.~$b$) is also a focal line of $T$ (resp.~$U$), we obtain $c=\overline{a}$ (resp.~$b=\overline{a}$). 
This yields a contradiction, since the lines $b$ and $c$ are skew. 
Therefore, the three lines $a,b,c$ are real. 

Next, we proceed with the classification. 
Proposition~C.\ref{propC: (1,2,2)} already describes the general class. 
Moreover, since all focal curves in the general class are lines, the following degenerations may occur:
\begin{enumerate}
\item Two focal lines approach each other along a common normal direction, until they intersect at a point. 
In this case, we obtain a degenerate line congruence $(C)$.
\item Two focal lines approach each other along a smooth quadric until they become infinitely near (flat limit), yielding a double line contained in a smooth quadric. 
In this case, we obtain a parabolic linear congruence $(A.3)$.
\end{enumerate} 
The lines $x$ and $y$ cannot degenerate as described in $1$, since this would imply that $a$, $b$, and $c$ are coplanar, yielding planar congruences $T$ and $U$. 
Similarly, $b$ (resp$.$ $c$) and $a$ cannot degenerate as in $1$, since then $x$ and $y$ would be coplanar, yielding a planar congruence $S$. 
However, the lines $b$ and $c$ may degenerate to a configuration in which they intersect, since they are not the focal lines of any line congruence and do not give rise to planar line congruences.
This occurs in classes a$.$5) and a$.$8). 
Therefore, the possible degenerations are precisely those listed in the statement. 
The fact that all these possibilities occur as the parametric line congruences of a trilinear birational map of type $(1,2,2)$ is ensured by \cite[Theorem $5.6$]{trilinear}.
\end{proof} 

In the following Example \ref{example: nonreal focal lines}, we provide an explicit example of a real trilinear birational map in class $b.1)$. 
Remarkably, the classes $b)$ in Theorem D$.$\ref{theoD: (1,2,2)} are the only ones that exhibit non-real focal varieties in the entire classification. 

\begin{example}
\label{example: nonreal focal lines}
Consider the real rational map
\begin{eqnarray*}
\phi : (\P_\C^1)^3
&\dashrightarrow &\P_\C^3 
\\[2pt] 
\nonumber
(s_0:s_1)
\times 
(t_0:t_1)
\times 
(u_0:u_1)
&\mapsto 
& 
(f_0:f_1:f_2:f_3)
\end{eqnarray*} 
whose defining trilinear polynomials are 
\begin{align*}
f_0 = & \,
s_0t_0u_1 + s_1t_0u_1 + s_0t_1u_1
\ ,& 
f_1 = & \,
s_1t_0u_1 + s_0t_1u_1 + 2s_1t_1u_1
\ ,
\\[4pt]
f_2 = & \,
-s_0t_1u_0 - s_1t_1u_0 + s_0t_0u_1 + s_1t_0u_1
\ ,& 
f_3 = & \,
-s_1t_1u_0 + s_1t_0u_1
\ .
\end{align*}
\noindent 
The map is birational of type $(1,2,2)$, as it admits the inverse
\begin{eqnarray*}
\phi^{-1} : \P_\C^3
&\dashrightarrow & (\P_\C^1)^3
\\[2pt] 
\nonumber
(x_0:x_1:x_2:x_3)
&\mapsto 
& 
(\sigma_0:\sigma_1)
\times 
(\tau_0:\tau_1)
\times 
(\upsilon_0:\upsilon_1)
\end{eqnarray*} 
given by
\begin{align*}
\sigma_0 = & \,
x_2 - x_3
\ ,& 
\sigma_1 = & \,
x_3
\ ,
\\[4pt]
\tau_0 = & \,
x_0x_2 - x_1x_2 + x_0x_3 + x_1x_3
\ ,& 
\tau_1 = & \,
x_1x_2 - x_0x_3
\ ,
\\[4pt]
\upsilon_0 = & \,
x_0x_2 - x_1x_2 - x_2^2 + x_0x_3 + x_1x_3 - x_3^2
\ ,& 
\upsilon_1 = & \,
x_1x_2 - x_0x_3
\ .
\end{align*}
\noindent 
We obtain the following:
\begin{enumerate}
\item The focal lines of $S$ are the skew complex-conjugate lines
$$
x\equiv (i x_0 - x_1, i x_2 - x_3)
\ \ ,\ \ 
y\equiv (i x_0 + x_1, i x_2 + x_3)
\ .
$$
Hence the parametric line congruence $S$ is elliptic linear $(A.2)$.

\item The focal lines of $T$ are the skew real lines
$$
a\equiv (x_2,x_3)
\ \ ,\ \ 
c\equiv  (x_0-x_2,x_1-x_3)
\ .
$$
Hence $T$ is hyperbolic linear $(A.1)$.

\item The focal lines of $U$ are the skew real lines
$$
a\equiv (x_2,x_3)
\ \ ,\ \ 
b\equiv (x_0,x_1)
\ .
$$
Hence $U$ is hyperbolic linear $(A.1)$.
\end{enumerate}

Therefore $\phi$ belongs to class $b.1)$ in Theorem~D.\ref{theoD: (1,2,2)}.
\end{example}

\subsection{Real Line Congruences of Trilinear Birational Maps of Type $(2,2,2)$} 

Finally, we focus on real trilinear birational maps of type $(2,2,2)$. 
In this case, the inverse is quadratic on each factor of $\mathbb{P}^1_{\mathbb{C}} \times \mathbb{P}^1_{\mathbb{C}} \times \mathbb{P}^1_{\mathbb{C}}$. 

As usual, we begin proving explicit formulas for the syzygies of a general birational map. 

\begin{mylemmaA}
\label{lemmaA: (2,2,2)}
Let $\phi$ be a general trilinear birational map of type $(2,2,2)$. 
The defining polynomials $\Phi = (f_0,f_1,f_2,f_3)$ admit syzygies of the form 
\begin{gather}
\alpha_1 
= 
b_1 c_0 \, B - b_0 c_1 \, C
\ \ ,\ \ 
\beta_1 
= 
a_1 c_0 \, A - a_0 c_1 \, C
\ \ ,\ \ 
\gamma_1 
= 
a_1 b_0 \, A - a_0 b_1 \, B
\ \ ,\\[4pt]
\alpha_2 
= 
\omega_1 b_1 c_0 \, B + (\omega_2 b_1 c_0 + \omega_3 b_0 c_1) \, A - b_1 c_1 \, D
\ \ ,\\[4pt]
\beta_2 
= 
\omega_2 a_1 c_0 \, A + (\omega_1 a_1 c_0 + \omega_3 a_0 c_1) \, B - a_1 c_1 \, D
\ \ ,\\[4pt]
\gamma_2 
= 
\omega_3 a_1 b_0 \, A + (\omega_1 a_1 b_0 + \omega_2 a_0 b_1) \, C - a_1 b_1 \, D
\ \ \phantom{,}
\end{gather} 
for some linear $a_i = a_i(s_0,s_1)$, $b_j = b_j(t_0,t_1)$, $c_k = c_k(s_0,s_1) \in R$ 
and linearly independent $A,B,C,D\in (\C^4)^\vee$.  
\end{mylemmaA}

\begin{proof}
It follows from \cite[Corollary $6.9$]{trilinear_siaga} that we can find linearly independent $A,B,C,D\in (\C^4)^\vee$ satisfying 
\begin{gather*}
\langle 
A
, 
\Phi
\rangle
= 
a_0 b_1 c_1 
\ \ ,\ \  
\langle 
B
, 
\Phi
\rangle
= 
a_1 b_0 c_1 
\ \ ,\ \ 
\langle 
C
, 
\Phi
\rangle
= 
a_1 b_1 c_0 
\ \ ,\\[4pt]   
\langle 
D
, 
\Phi
\rangle
= 
\omega_1 a_1 b_0 c_0 
+ 
\omega_2 a_0 b_1 c_0 
+ 
\omega_3 a_0 b_0 c_1  
\ \ ,  
\end{gather*}
for some linear $a_i=a_i(s_0,s_1)$, 
$b_j = b_j(t_0,t_1)$, 
$c_k = c_k(u_0,u_1)$ and nonzero $\omega_1,\omega_2,\omega_3\in \C$. 
Thus, it is straightforward to verify that the relations in the statement are syzygies of $\Phi$. 
\end{proof}

\begin{mycorB}
\label{corB: (2,2,2)}
With the notation of Lemma A$.$\ref{lemmaA: (2,2,2)},  the parametric line congruences admit the rational parameterizations $\P_\C^1\times\P_\C^1 \dashrightarrow Q$ given by  
\begin{eqnarray}
\label{eq: lc 1 (2,2,2)}
S \equiv &  
\alpha_1 \wedge \alpha_2 & = 
(b_1 c_0 \, B - b_0 c_1 \, C)
\wedge 
(\omega_1 b_1 c_0 \, B 
+ 
(\omega_2 b_1 c_0 
+ 
\omega_3 b_0 c_1) \, A 
- 
b_1 c_1 \, D
)
\ ,
\\[4pt]
\label{eq: lc 2 (2,2,2)}
T \equiv &  
\beta_1 \wedge \beta_2 & = 
(a_1 c_0 \, A - a_0 c_1 \, C)
\wedge 
(\omega_2 a_1 c_0 \, A 
+ 
(\omega_1 a_1 c_0 
+ 
\omega_3 a_0 c_1) \, B 
- 
a_1 c_1 \, D
)
\ ,
\\[4pt]
\label{eq: lc 3 (2,2,2)}
U \equiv &  
\gamma_1 \wedge \gamma_2 & = 
(a_1 b_0 \, A - a_0 b_1 \, B)
\wedge 
(\omega_3 a_1 b_0 \, A 
+ 
(\omega_1 a_1 b_0 
+ 
\omega_2 a_0 b_1) \, C 
- 
a_1 b_1 \, D
)
\ .
\end{eqnarray}
\end{mycorB}

\begin{mypropC} 
\label{propC: (2,2,2)}
Let $\phi$ be a general trilinear birational map of type $(2,2,2)$. 
Then, we can find distinct lines $x,y,z\subset\P_\C^3$ and a plane conic $c\subset\P_\C^3$ satisfying:
\begin{enumerate}[itemsep=2pt]
\item[1.] The intersection $P = x \cap y \cap z$ is a point.  
\item[2.] The line $x$ (resp$.$ $y$ and $z$) intersects $c$ at a point. 
\end{enumerate}
such that:
\begin{enumerate}[itemsep = 2pt]
\item[a)] The focal curves of $S$ are $x$ and $c$. 
\item[b)] The focal curves of $T$ are $y$ and $c$. 
\item[c)] The focal curves of $U$ are $z$ and $c$. 
\end{enumerate} 
\end{mypropC}

\noindent See also Fig.~\ref{fig:1658}.

\begin{figure}\centering\setlength{\unitlength}{1mm}\footnotesize
  \begin{picture}(0,0)
    \put(11,0){$x$}\put(30,13){$y$}\put(54,0){$z$}\put(0,7){$c$}
    \put(22,10){$S$}\put(35,24){$T$}\put(38,10){$U$}
  \end{picture}
  \includegraphics[width=.4\textwidth]{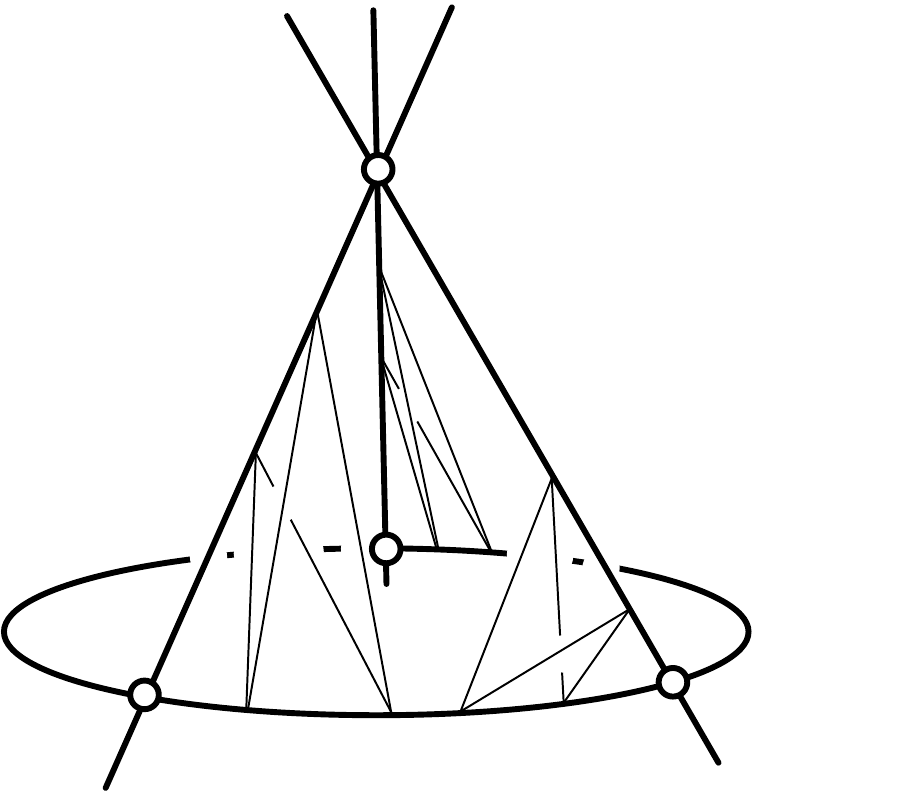}
  \caption{Focal curves (thick lines) and line congruences (thin lines) for type $(2,2,2)$, class 1.}
  \label{fig:1658}
\end{figure}
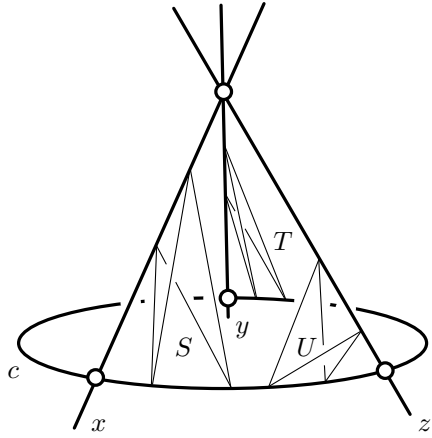

\begin{proof}
We start from the parameterizations of Corollary B$.$\ref{corB: (2,2,2)}. 
First, define the lines $x,y,z\subset \P_\C^3$ with Plücker coordinates
\begin{equation}
\label{eq: focal lines (2,2,2)}
x \equiv B \wedge C 
\ , \ 
y \equiv A \wedge C
\ , \ 
z \equiv A \wedge B
\end{equation} 
which are all nonzero by linear independence. 
By definition, the three lines intersect at the point $P$ dual to $A\wedge B\wedge C$, as explained in Section \ref{subsec: duality}. 
It is straightforward that $x$ intersects every line $\ell\in S$ since $x\wedge S = 0$ because $x \wedge \alpha_1 = x \wedge (b_1 c_0 \, B - b_0 c_1 \, C) = 0$. 
The obvious argument similarly proves that $y$ (resp$.$ $z$) intersects every line $\ell\in T$ (resp$.$ $\ell\in U$). 

Next, we prove that every $\ell \in S$ meets the plane conic $c\subset\P_\C^3$ defined by 
$$
D = 0
\ \ ,\ \ 
Q
= 
\omega_1 
BC 
+ 
\omega_2 
AC 
+
\omega_3 
AB
= 0
\ . 
$$
The conic $c$ is smooth since, by linear independence, $D(P)\not=0$.
Geometrically, the plane $D=0$ does not contain the unique singular point $P$ of $Q = 0$. 
For simplicity, we apply the linear change of variables 
$$
A \mapsto x_1 
\ \ ,\ \ 
B \mapsto x_2 
\ \ ,\ \ 
C \mapsto x_3 
\ \ ,\ \ 
D \mapsto x_0
\ \ ,
$$
which sends the parametrization \eqref{eq: lc 1 (2,2,2)} of $S$ to 
\begin{align*}
S 
&
\equiv 
(b_1 c_0 \, x_2 - b_0 c_1 \, x_3)
\wedge 
(\omega_1 b_1 c_0 \, x_2 
+ 
(\omega_2 b_1 c_0 
+ 
\omega_3 b_0 c_1) \, x_1 
- 
b_1 c_1 \, x_0
) 
\\[4pt]
& 
= 
- 
b_1 c_0
(\omega_2 b_1 c_0 
+ 
\omega_3 b_0 c_1)
\, 
(x_1\wedge x_2) 
+ 
b_1^2 c_0 c_1 
\, 
(x_0\wedge x_2) 
+ \\[4pt]
&
\phantom{=}
\omega_1 b_0 b_1 c_0 c_1 
\,
(x_2 \wedge x_3) 
+ 
b_0 c_1 
(\omega_2 b_1 c_0 
+ 
\omega_3 b_0 c_1)
\,
(x_1\wedge x_3)
- 
b_0 b_1 c_1^2 
\, 
(x_0 \wedge x_3)
\ .
\end{align*}
In particular, the dual Plücker coordinates (see Section \ref{subsec: duality}) write as 
\begin{align*}
S 
&
\equiv 
- 
b_1 c_0
(\omega_2 b_1 c_0 
+ 
\omega_3 b_0 c_1)
\, 
(e_0\wedge e_3) 
-
b_1^2 c_0 c_1 
\, 
(e_1\wedge e_3) 
+ \\[4pt]
&
\phantom{=}
\omega_1 b_0 b_1 c_0 c_1 
\,
(e_0 \wedge e_1) 
- 
b_0 c_1 
(\omega_2 b_1 c_0 
+ 
\omega_3 b_0 c_1)
\,
(e_0\wedge e_2)
- 
b_0 b_1 c_1^2 
\, 
(e_1 \wedge e_2)
\ .
\end{align*}
Thus, the intersection of the generic line in $S$ with $x_0 = 0$ is computed as the contraction of $S$ in by $x_0$, namely 
$$
\iota_{x_0}(S)
= 
-b_1c_0 
(\omega_2 b_1 c_0 + \omega_3 b_0 c_1) 
\, 
e_3 
+ 
\omega_1 b_0 b_1 c_0 c_1 
\, 
e_1 
- 
b_0 c_1
(\omega_2 b_1 c_0 + \omega_3 b_0 c_1) 
\, 
e_2
\ .
$$
In particular, $\iota_{x_0}(S)$ satisfies  $Q = 0$. 
Hence, $\ell\cap\{x_0=0\}\in c$ and every $s$-line meets $c$. 
Hence, $S$ must be the quadratic line congruence with focal curves $x$ and $c$. 
With the obvious modifications we conclude that every $\ell\in T$ (resp$.$ $\ell\in U$) also intersects $c$. 
The statement follows. 
\end{proof}

\begin{mytheoremD} 
\label{theoD: (2,2,2)}
Let $\phi$ be a general real trilinear birational map of type $(2,2,2)$, and let $x,y,z,c\subset \P_\C^3$ be the curves in Proposition C$.$\ref{propC: (2,2,2)}. 
Then, the four focal curves are real. 

Moreover, every (possibly degenerate) real parametric line congruence in case a) falls (up to permutation of the factors of $\P_\C^1\times\P_\C^1\times\P_\C^1$) into one of the following classes: 
\begin{enumerate}
\item[1)] (General) The intersection $P = x\cap y\cap z$ is a point outside the plane supporting $c$:  
\begin{itemize}
\item $S$ is quadratic $(B)$ with focal curves $x$ and $c$.  
\item $T$ is quadratic $(B)$ with focal curves $y$ and $c$.     
\item $U$ is quadratic $(B)$ with focal curves $z$ and $c$.   
\end{itemize}
\item[2)] The intersection $P = x\cap y\cap z$ is a point supported on $c$:   
\begin{itemize}
\item $S$ is quadratic $(B)$ with focal curves $x$ and $c$.  
\item $T$ is quadratic $(B)$ with focal curves $y$ and $c$.     
\item $U$ is quadratic $(B)$ with focal curves $z$ and $c$.   
\end{itemize}
\end{enumerate}
\end{mytheoremD} 

\begin{proof}
If $\phi$ is general, Proposition C$.$\ref{propC: (2,2,2)} readily implies that the three parametric line congruences are quadratic $(B)$.  
Since $\phi$ is real, Corollary \ref{corollary: focal real} readily implies that the focal curves $a,b,c,x,y\subset \P_\C^3$ are all real. 

Regarding the classification, Proposition B$.$\ref{propC: (2,2,2)} provides the general class.  
Moreover, the computation of the inverse in the two classes listed in \cite[Theorem $5.8$]{trilinear} shows that the common focal conic $c$ to the three parametric line congruences is smooth. 
Moreover, no two of the lines $x,y,z$ can degenerate to the same line, as this would force two of the parametric line congruences to coincide. 
Finally, if in the general class the intersection point $P = x \cap y \cap z$ degenerates to a point on the plane supporting $c$, then $P$ must lie on $c$, since otherwise the corresponding parametric line congruences would be planar. 
Hence, the only possible degeneration is $P\in c$, namely the three lines $x,y,z$ meet $c$ at the same point $P$.  
\end{proof}

\section{Conclusion}
\label{sec:conl}

The main contribution of the paper is the classification of all possible real line congruences arising from the parametric lines of real trilinear birational maps $\phi:(\P_\C^1)^3\dashrightarrow \P_\C^3$. 
A notable feature is that non-real focal varieties may occur, as in the case of the classes b) in Theorem D$.$~\ref{theoD: (1,2,2)}.  
Example~\ref{example: nonreal focal lines} provides an explicit real trilinear birational map whose associated parametric line congruence $S$ has focal curves given by two skew complex-conjugate lines, which have no real points.

\subsection*{Acknowledgment}
This research was funded in whole or in part by the Austrian Science Fund (FWF) 10.55776/PIN6740223 and 10.55776/I6779. 
The second author is thankful to Javier Sendra-Arranz for useful suggestions and discussions.

\end{document}